\documentclass[10pt]{amsart}

\usepackage[margin=1in]{geometry}
\usepackage{amsthm}
\usepackage{amsmath}
\usepackage{mathrsfs}
\usepackage{amssymb}
\usepackage[dvipsnames]{xcolor}
\usepackage{tikz-cd}
\usepackage{enumitem}
\usepackage{hyperref}
\hypersetup{colorlinks = true,
	linkcolor = MidnightBlue,
	urlcolor = BrickRed,
	citecolor = MidnightBlue}

\title[On deformations and $K$-cohomology]{On the algebraizability of formal deformations in $K$-cohomology}
\author{Eoin Mackall}
\email{eoinmackall \emph{at} gmail.com}
\urladdr{\url{www.eoinmackall.com}}
\date{\today}

\newtheorem{thm}{Theorem}[section]

\newtheorem{prop}[thm]{Proposition}
\newtheorem{cor}[thm]{Corollary}
\newtheorem{lem}[thm]{Lemma}

\theoremstyle{definition}
\newtheorem{defn}[thm]{Definition}
\newtheorem{exmp}[thm]{Example}
\newtheorem{rmk}[thm]{Remark}

\newcommand{\HH}{\mathrm{H}}
\newcommand{\CH}{\mathrm{CH}}

\begin{document}
	\maketitle
	\begin{abstract}
We show that algebraizability of the functors $R^1\pi_*\mathcal{K}^M_{2,X}$ and $R^2\pi_*\mathcal{K}^M_{2,X}$ is a stable birational invariant for smooth and proper varieties $\pi:X\rightarrow k$ defined over an algebraic extension $k$ of $\mathbb{Q}$. The same is true for the \'etale sheafifications of these functors as well.

To get these results we introduce a notion of relative $K$-homology for schemes of finite type over a finite dimensional, Noetherian, excellent base scheme over a field. We include this material in an appendix.
\end{abstract}

\section{Introduction}
Let $k/\mathbb{Q}$ be an algebraic field extension and let $\pi:X\rightarrow k$ be a smooth and proper $k$-variety. In \cite{MR371891}, Bloch gives a cohomological criterion for the pro-representability of the higher direct image $R^i\pi_*\mathcal{K}_{2,X}$ of the $2$nd $K$-theory sheaf $\mathcal{K}_{2,X}$ on the big Zariski site of $X$. Specifically, if $F_X^i:\mathsf{Art}_{k}\rightarrow \mathsf{Ab}$ is defined on the category $\mathsf{Art}_k$ of artinian local $k$-algebras $(A,\mathfrak{m}_A)$ with residue field $A/\mathfrak{m}_A\cong k$ by \begin{equation}F_X^i(A)=\mathrm{ker}\left(\mathrm{H}^i(X_A,\mathcal{K}_{2,X_A})\rightarrow \mathrm{H}^i(X,\mathcal{K}_{2,X})\right)\end{equation} then Bloch shows that $F_X^i$ is pro-representable if and only if $\mathrm{H}^i(X,\mathcal{O}_X)=\mathrm{H}^{i+1}(X,\mathcal{O}_X)=0$.

The cohomology of the structure sheaf $\mathcal{O}_X$ is well-known to be a stable birational invariant of $X$, \cite{MR2923726}. Combined with Bloch's criterion, this implies that for $X,Y$ smooth, proper, and stably birational varieties over $k$, the functor $R^i\pi_*\mathcal{K}_{2,X}$ is pro-representable (at any $L$-point for any finite field extension $L/k$) if and only if $R^i\phi_{*}\mathcal{K}_{2,Y}$ is as well, where $\phi:Y\rightarrow k$ is the structure map of $Y$. We show in Theorems \ref{thm: birational} and \ref{thm: birational2} below that, in either the case $i=1$ or $i=2$, the \textit{algebraizability} of the functor $R^i\pi_{*}\mathcal{K}_{2,X}$ is also a stable birational invariant of $X$. (Below we will replace $\mathcal{K}_{2,X}$ with the Milnor $K$-theory sheaf $\mathcal{K}^M_{2,X}$. The two are canonically isomorphic, so the difference is mostly notational.)

The proofs of Theorems \ref{thm: birational} and \ref{thm: birational2} rely on the use of weak factorization, as well as a description of both $K$-cohomology groups and the algebraic de Rham cohomology groups of projective bundles and blow ups. We include, in an appendix, an overview of the necessary results on $K$-cohomology which also corrects some minor errors in the development of these groups available in the literature. Throughout this article we also point out how to apply these methods to achieve the same result for the \'etale sheafification $(R^i\pi_{*}\mathcal{K}^M_{2,X})_{\acute{e}t}$ of the higher direct image functor $R^i\pi_{*}\mathcal{K}^M_{2,X}$ for $i=1$ or $i=2$.

The functor $R^n\pi_*\mathcal{K}^M_{n,X}$ can be thought of as a direct generalization to ``codimension-$n$ cycle classes" of the Zariski sheafification of the Picard functor $\mathrm{Pic}_{X/k,(Zar)}=R^1\pi_*\mathcal{K}^M_{1,X}$, the latter of which is well-known to be representable for sufficiently nice varieties $X$. The functors $R^i\pi_*\mathcal{K}^M_{n,X}$ for $i\neq n$ can also be thought of as a ``higher generalization" of the Picard functor but, varying the \textit{weight} and \textit{codimension} independently. Algebraizability of the functors $R^i\pi_*\mathcal{K}^M_{n,X}$ essentially means that for any finite field extension $L/k$ and for any $L$-point $\xi$ of $R^i\pi_*\mathcal{K}^M_{n,X}$, one can find a pointed finite type $k$-scheme $(Z,z)$ with residue field $k(z)\cong L$ and a deformation of $\xi$ on $Z$ which formally locally around $z$ is the universal deformation of $\xi$.

While not implying representability directly, it is known by \cite[Theorem 4.1]{MR0260746} that if $(R^i\pi_*\mathcal{K}^M_{n,X})_{\acute{e}t}$ is both algebraizable and \textit{relatively representable} then $(R^i\pi_*\mathcal{K}^M_{n,X})_{\acute{e}t}$ is in fact representable. One is naturally led to wonder if $(R^i\pi_*\mathcal{K}^M_{n,X})_{\acute{e}t}$ is representable in any interesting examples.

\section{Background}\label{sec:back}
Let $\pi:X\rightarrow k$ be the structure map of an arbitrary $k$-variety $X$. We write $\mathcal{K}^M_{n,X}$ for the $n$th Milnor $K$-theory sheaf on the big Zariski site $\mathsf{Sch}/X$ (since we will assume that $k$ has characteristic zero below, and also $n=2$, one could use the isomorphic Quillen $K$-theory sheaf to the same effect). We write $R^i\pi_*\mathcal{K}^M_{n,X}$ for the associated Zariski sheaf on $\mathsf{Sch}/k$. For any $k$-scheme $T$, we have \begin{equation}R^i\pi_*\mathcal{K}^M_{n,X}(T)=\mathrm{H}^0(T,R^i\pi_{T*}\mathcal{K}^M_{n,X_T})\end{equation} where $\pi_{T}:X_T\rightarrow T$ is the map obtained from $\pi$ by base change. 

\begin{rmk}\label{rmk:local}
	Let $(A,\mathfrak{m}_A)$ and $(B,\mathfrak{m}_B)$ be local $L$-algebras for any field extension $L/k$. If $T=\mathrm{Spec}(A)$ and $T'=\mathrm{Spec}(B)$, then for any morphism $\rho:T'\rightarrow T$ the induced pullback map \begin{equation} \mathrm{H}^0(T,R^i\pi_{T*}\mathcal{K}^M_{n,X_T})\rightarrow \mathrm{H}^0(T',R^i\pi_{{T'}*}\mathcal{K}^M_{n,X_{T'}})\end{equation} is equivalent to the natural map of groups $\mathrm{H}^i(X_T,\mathcal{K}^M_{n,X_T})\rightarrow \mathrm{H}^i(X_{T'},\mathcal{K}^M_{n,X_{T'}})$.
\end{rmk}

We say that an abelian sheaf $\mathcal{F}$ on the big Zariski site $\mathsf{Sch}/k$ is \textit{pro-representable} if it is so around any $L$-point $\xi\in \mathcal{F}(\mathrm{Spec}(L))$ for any finite field extension $L/k$ (see \cite{MR217093} for the definition of a pro-representable functor of artin local $L$-algebras). Since $\mathcal{F}$ is a sheaf of groups, it is enough to check that this is true only around the identity element of $\mathcal{F}(\mathrm{Spec}(L))$. In this case, being pro-representable is equivalent to the existence of a complete local noetherian $L$-algebra $(R,\mathfrak{m}_R)$ such that $R/\mathfrak{m}_R\cong L$, the quotient $R/\mathfrak{m}_R^t$ is a local artinian $L$-algebra for all $t\geq 2$, and such that there is a natural (in $A$) isomorphism \begin{equation}\mathrm{Hom}_{\text{local $L$-alg.}}(R,A)\cong \mathrm{ker}\left(\mathcal{F}(\mathrm{Spec}(A))\rightarrow \mathcal{F}(\mathrm{Spec}(L))\right)\end{equation} for any local artinian $L$-algebra $(A,\mathfrak{m}_A)$. From now on, for any pointed $L$-scheme $(X,x)$ with residue field $L(x)\cong L$ or for any $L$-algebra $A$ with maximal ideal $\mathfrak{m}$ admitting an isomorphism $A/\mathfrak{m}\cong L$, we write \begin{equation} \mathrm{Def}(\mathcal{F})(X)=\mathrm{ker}\left(\mathcal{F}(X)\rightarrow \mathcal{F}(x)\right) \quad \mbox{or}\quad \mathrm{Def}(\mathcal{F})(A)=\mathrm{ker}\left(\mathcal{F}(\mathrm{Spec}(A))\rightarrow \mathcal{F}(\mathrm{Spec}(A/\mathfrak{m}))\right)\end{equation} for the sets of \textit{$X$-deformations, or $A$-deformations, of the identity of $\mathcal{F}(\mathrm{Spec}(L))$} respectively.

We say that an abelian sheaf $\mathcal{F}$ which is pro-representable is \textit{effectively pro-representable} if it is so around any $L$-point $\xi\in \mathcal{F}(\mathrm{Spec}(L))$ for any finite field extension $L/k$. Again, since $\mathcal{F}$ is a sheaf of groups, it is enough to check this condition around the identity element of the group $\mathcal{F}(\mathrm{Spec}(L))$. Then $\mathcal{F}$ is effectively pro-representable if the formal deformation corresponding to the identity of $R$ is effective, i.e.\ if there exists an element $u_R\in \mathrm{Def}(\mathcal{F})(R):=\mathrm{ker}(\mathcal{F}(\mathrm{Spec}(R)\rightarrow \mathcal{F}(\mathrm{Spec}(L)))$ mapping to the identity of $R$ under the map \begin{equation}\mathrm{Def}(\mathcal{F})(R)\rightarrow \mathrm{Hom}_{\text{local $L$-alg.}}(R,R)\cong \varprojlim_t \mathrm{Def}(\mathcal{F})(R/\mathfrak{m}_R^t).\end{equation} For any complete noetherian local $L$-algebra $(S,\mathfrak{m}_S)$ with residue field $S/\mathfrak{m}_S\cong L$ such that $S/\mathfrak{m}_S^t$ is an artianian local $L$-algebra for all $t\geq 2$, we will write $\overline{\mathrm{Def}}(\mathcal{F})(S):=\varprojlim_{t} \mathrm{Def}(\mathcal{F})(S/\mathfrak{m}_S^t)$ for the set of \textit{formal $S$-deformations of the identity of $\mathcal{F}(\mathrm{Spec}(L))$}.

\begin{rmk}\label{rmk:changingrings} Now let $(S,\mathfrak{m}_S)$ be any complete noetherian local $L$-algebra such that $S/\mathfrak{m}_S\cong L$ and $S/\mathfrak{m}_S^t$ is an artinian local $L$-algebra for all $t\geq 2$. If $\mathcal{F}$ is an effectively pro-representable functor, then any formal $S$-deformation $\bar{\xi}\in \overline{\mathrm{Def}}(\mathcal{F})(S)$ is effective, i.e.\ there exists an element $\xi\in \mathrm{Def}(\mathcal{F})(S)$ mapping to $\bar{\xi}$. Indeed, there exists a local $L$-algebra homomorphism $f:R\rightarrow S$ so that $\overline{\mathcal{F}(f)}(\mathrm{id}_R)=\bar{\xi}$ in the diagram \begin{equation}\begin{tikzcd}
	\mathrm{Def}(\mathcal{F})(S)\arrow{r} & \overline{\mathrm{Def}}(\mathcal{F})(S)\\
	\mathrm{Def}(\mathcal{F})(R)\arrow{r}\arrow["\mathcal{F}(f)"]{u} & \overline{\mathrm{Def}}(\mathcal{F})(R)\arrow["\overline{\mathcal{F}(f)}"]{u}
\end{tikzcd}\end{equation} by the universal property of the $L$-algebra $(R,\mathfrak{m}_R)$.
\end{rmk}

We say that an abelian sheaf $\mathcal{F}$ which is effectively pro-representable is \textit{algebraizable} if it is so around any $L$-point $\xi\in \mathcal{F}(\mathrm{Spec}(L))$ for any finite field extension $L/k$ or, equivalently, if it is so around the identity for all such fields $L/k$. So an abelian sheaf $\mathcal{F}$ is algebraizable if there exists a finite type $k$-scheme $X$, a closed point $x\in X$ with residue field $k(x)\cong L$, an element $u_X\in \mathrm{Def}(\mathcal{F})(X)$, and an isomorphism $\widehat{\mathcal{O}}_{X,x}\cong R$ of local $L$-algebras such that $u_X$ induces the identity of $R$ under the canonical map \begin{equation}\mathrm{Def}(\mathcal{F})(X)\rightarrow \overline{\mathrm{Def}}(\mathcal{F})(\widehat{\mathcal{O}}_{X,x})\cong\mathrm{Hom}_{\text{local $L$-alg.}}(R,\widehat{\mathcal{O}}_{X,x})\cong \mathrm{Hom}_{\text{local $L$-alg.}}(R,R).\end{equation}

Since $\mathcal{K}^M_{n,X}$ is locally of finite presentation for all $n\in \mathbb{Z}$ by construction, it follows from \cite[\href{https://stacks.math.columbia.edu/tag/0A37}{Tag 0A37}]{stacks-project} that $R^i\pi_{*}\mathcal{K}^M_{n,X}$ is locally of finite presentation whenever $\pi$ is quasi-compact and quasi-separated. Similarly, if $\pi$ is quasi-compact and quasi-separated then the sheafifications \begin{equation}\left(R^i\pi_{*}\mathcal{K}^M_{n,X}\right)_{fppf}\quad \mbox{and}\quad \left(R^i\pi_{*}\mathcal{K}^M_{n,X}\right)_{\acute{e}t}\end{equation} for the fppf and \'etale topologies are locally of finite presentation by the proof of \cite[\href{https://stacks.math.columbia.edu/tag/049O}{Tag 049O}]{stacks-project}. Thus, by Artin's Algebraization theorem \cite[Theorem 1.6]{MR0260746}, for fixed $i$ and $n$ the functor $R^i\pi_*\mathcal{K}^M_{n,X}$ is algebraizable if it is effectively pro-representable.

Bloch shows in \cite[Theorem 0.2]{MR371891} that for a smooth and proper variety $X$ defined over an algebraic extension $k/\mathbb{Q}$, the functor $R^i\pi_*\mathcal{K}^M_{2,X}$ is pro-representable if and only if $\mathrm{H}^i(X,\mathcal{O}_X)=\mathrm{H}^{i+1}(X,\mathcal{O}_X)=0$. Moreover, Bloch gives a canonical isomorphism $\mathrm{Def}(R^i\pi_*\mathcal{K}^M_{2,X})(A)\cong \mathrm{H}^i(X_L,\Omega^1_{X_L})\otimes_L \mathfrak{m}_A$ for any artin local $L$-algebra $(A,\mathfrak{m}_A)$ with residue field $A/\mathfrak{m}_A\cong L$ and for any finite field extension $L/k$. In this setting, the representing ring $R$ as above is canonically isomorphic to a formal completion of the symmetric $L$-algebra of the dual $L$-vector space $V^\vee$ where $V=\mathrm{H}^i(X_L,\Omega^1_{X_L})$.

\section{In the case when $i=2$}
Throughout this section we fix an algebraic extension $k/\mathbb{Q}$. We first deal with the case $i=2$ where our functor $R^2\pi_*\mathcal{K}^M_{2,X}$ is comparable to the familiar Chow group of codimension two cycles.

\begin{rmk}\label{rmk: gerstform}
 Let $R=k[[t_1,...,t_r]]$ be a ring of formal power series over $k$ in finitely many variables. Let $X$ be a smooth scheme, geometrically connected and quasi-compact over $k$. Then the product $X'=X\times_k \mathrm{Spec}(R)$ is a connected, regular, and excellent Noetherian scheme of finite Krull dimension. Indeed, if $U=\mathrm{Spec}(B)$ is an affine open subset of $X$, then $B$ is a finitely generated $k$-algebra and $B\otimes_k R$ is Noetherian by Hilbert's basis theorem. Thus $X'$ is locally Noetherian and the projection $X'\rightarrow \mathrm{Spec}(R)$ is quasi-compact, so $X'$ is locally Noetherian and quasi-compact, hence Noetherian.
	
	Since $X$ is geometrically connected and $R$ is integral, it follows that $X'$ is connected, \cite[\href{https://stacks.math.columbia.edu/tag/0385}{Tag 0385}]{stacks-project}. One can bound the dimension of $X'$ in terms of the dimension of $X$ and $\mathrm{Spec}(R)$ using \cite[\href{https://stacks.math.columbia.edu/tag/0AFF}{Tag 0AFF}]{stacks-project} and \cite[\href{https://stacks.math.columbia.edu/tag/04MU}{Tag 04MU}]{stacks-project}. Excellence of $X'$ follows from \cite[\href{https://stacks.math.columbia.edu/tag/07QW}{Tag 07QW}]{stacks-project}.
	
	To see that $X'$ is regular, let $x\in X'$ be a point. Let $y\in \mathrm{Spec}(R)$ be the image of the point $x$ under the second projection $X'\rightarrow \mathrm{Spec}(R)$, and write $\mathfrak{m}_y\subset \mathcal{O}_{\mathrm{Spec}(R),y}$ for the maximal ideal of the local ring at $y$. The induced ring map $\mathcal{O}_{\mathrm{Spec}(R),y}\rightarrow \mathcal{O}_{X',x}$ is both local and, since $X$ is smooth, flat. Since there exists a canonical isomorphism \[\mathcal{O}_{X',x}/\mathfrak{m}_y\mathcal{O}_{X',x}\cong \mathcal{O}_{X'_y, x},\] between the fiber over $y$ of the local ring of $x$ in $X'$ and the local ring of $x$ in the fiber over $y$, and since $X$ is smooth, the local ring $\mathcal{O}_{X'_y,x}\cong \mathcal{O}_{X',x}/\mathfrak{m}_y\mathcal{O}_{X',x}$ is regular. As $\mathcal{O}_{\mathrm{Spec}(R),y}$ is also regular, it follows from \cite[Theorem 23.7]{MR879273} that $\mathcal{O}_{X',x}$ is regular too.
	
	In particular, the Gersten conjecture is known to hold for the local rings of $X'$. (For Quillen's $K$-theory in the case of finite type $k$-algebras this is due to Quillen \cite{MR0338129} and for Quillen's $K$-theory for more general regular rings by Panin \cite{MR2024050}. For Milnor's $K$-theory this result is due to Kerz \cite{MR2461425} when the base field contains enough elements, e.g.\ if it is infinite.)
\end{rmk}

\begin{lem}\label{lem: probunlem}
	Fix an algebraic extension $k/\mathbb{Q}$ and let $\pi:X\rightarrow k$  be a smooth, proper, and geometrically connected scheme. Let $\mathcal{E}$ be a finite rank locally free sheaf on $X$ and $\varphi:\mathbb{P}(\mathcal{E})\rightarrow k$ be the structure map of the associated projective bundle. Then $R^2\pi_*\mathcal{K}^M_{2,X}$ is effectively pro-representable if and only if $R^2\varphi_* \mathcal{K}^M_{2,\mathbb{P}(\mathcal{E})}$ is effectively pro-representable.
	
	The above statement also holds replacing all Zariski sheaves with their \'etale sheafifications.
\end{lem}

\begin{proof}
	It suffices to work only over the base field $k$, noting that $k/\mathbb{Q}$ is arbitrary. Let $R=k[[t_1,...,t_r]]$ be a power series ring in finitely many variables with maximal ideal $\mathfrak{m}_R$. Set $S=\mathrm{Spec}(R)$ and $S_t=\mathrm{Spec}(R/\mathfrak{m}_R^t)$. Then, by Remark \ref{rmk: gerstform}, the Gersten conjecture holds for the local rings of both $X_S$ and $\mathbb{P}(\mathcal{E})_S$ so that \[\HH^2(X_S, \mathcal{K}^M_{2,X_S})\cong \mathrm{CH}^2(X_S)\quad \mbox{and} \quad \HH^2(\mathbb{P}(\mathcal{E})_S, \mathcal{K}^M_{2,\mathbb{P}(\mathcal{E})_S})\cong \mathrm{CH}^2(\mathbb{P}(\mathcal{E})_S)\] where for any smooth and geometrically connected $k$-scheme $Y$, we write $\CH^2(Y_S)$ for the (relative) Chow group of codimension-2 cycles on $Y_S$ modulo rational equivalence (cf.\ \cite[Chapter 20]{MR1644323}).
	
	Now suppose that $R^2\pi_*\mathcal{K}^M_{2,X}$ is pro-representable. By \cite[Theorem (0.2)]{MR371891}, there is a natural in $(A,\mathfrak{m}_A)$ isomorphism $\mathrm{Def}(R^2\pi_*\mathcal{K}^M_{2,X})(A)\cong \HH^2(X,\Omega_X^1)\otimes_k \mathfrak{m}_A$ for any artinian local $k$-algebra $(A,\mathfrak{m}_A)$. Thus there is a commutative diagram as below (using Remark \ref{rmk:local}).
	\begin{equation}\begin{tikzcd}[column sep = small] 0 \arrow{r} & \mathrm{Def}(R^2\pi_*\mathcal{K}^M_{2,X})(R) \arrow{r}\arrow{d} & \CH^2(X_S)\arrow{r}\arrow{d} & \CH^2(X)\arrow[equals]{d}\arrow{r} & 0  \\
			0\arrow{r} & \HH^2(X,\Omega_{X}^1)\otimes_k \mathfrak{m}_R \arrow{r} & \varprojlim_t \HH^2(X_{S_t}, \mathcal{K}^M_{2,X_{S_t}})\arrow{r} & \HH^2(X,\mathcal{K}^M_{2,X})\arrow{r} & 0
		\end{tikzcd}
	\end{equation}
	Similarly, there is a commutative diagram for $\mathbb{P}(\mathcal{E})$ as so:
	
\begin{equation}\begin{tikzcd}[column sep = small] 0\arrow{r} &\mathrm{Def}(R^2\varphi_*\mathcal{K}^M_{2,\mathbb{P}(\mathcal{E})})(R) \arrow{r}\arrow{d} & \CH^2(\mathbb{P}(\mathcal{E})_S)\arrow{r}\arrow{d} & \CH^2(\mathbb{P}(\mathcal{E}))\arrow[equals]{d}\arrow{r} & 0  \\
				0\arrow{r} & \HH^2(\mathbb{P}(\mathcal{E}),\Omega_{\mathbb{P}(\mathcal{E})}^1)\otimes_k \mathfrak{m}_R \arrow{r} &\varprojlim_t \HH^2(\mathbb{P}(\mathcal{E})_{S_t}, \mathcal{K}^M_{2,\mathbb{P}(\mathcal{E})_{S_t}})\arrow{r} & \HH^2(\mathbb{P}(\mathcal{E}),\mathcal{K}^M_{2,\mathbb{P}(\mathcal{E})}) \arrow{r} & 0.
		\end{tikzcd}\end{equation}
	
	The bottom rows in the diagrams above are canonically right-split by the pull-back along the structure maps for $X$ and $\mathbb{P}(\mathcal{E})$ respectively. Using the projection map $\mathbb{P}(\mathcal{E})\rightarrow X$, the two diagrams above can be compared (via pull-back) and this comparison respects these splittings. If $\mathcal{E}$ has $\mathrm{rk}(\mathcal{E})=1$, then the comparison is an isomorphism everywhere. Otherwise, if $\mathrm{rk}(\mathcal{E})>1$, then there is a diagram of corresponding cokernels:
	\begin{equation}\begin{tikzcd}[column sep = small] 0\arrow{r} &\mathrm{Def}(\mathrm{Pic}_{X/k,(Zar)})(R) \arrow{r}\arrow{d} & \mathrm{Pic}(X_S) \oplus A\arrow{r}\arrow{d} & \mathrm{Pic}(X)\oplus A\arrow[equals]{d}\arrow{r} & 0  \\
				0\arrow{r} & \HH^1(X,\mathcal{O}_X)\otimes_k \mathfrak{m}_R \arrow{r} &\varprojlim_t \HH^1(X_{S_t}, \mathcal{K}^M_{1,X_{S_t}})\arrow{r} & \HH^1(X,\mathcal{K}^M_{1,X}) \arrow{r} & 0.
		\end{tikzcd}\end{equation}
Here the nonzero column on the right (and the middle object in the top row) can be identified with use of the projective bundle formula (and we have $A=0$ if $\mathrm{rank}(\mathcal{E})=2$ and $A=\mathbb{Z}$ if $\mathrm{rank}(\mathcal{E})\geq 3$); the identification of the top-left object follows from this. The identification of the bottom-left object seems to be well-known, and the middle term in the bottom row can be identified using these facts together with the splittings of the bottom rows of the two previous diagrams.
	
	Altogether, this gives a commutative ladder with exact rows:
	\begin{equation}\begin{tikzcd}[column sep = small] 0\arrow{r} & \mathrm{Def}(R^2\pi_*\mathcal{K}^M_{2,X})(R)\arrow{r}\arrow{d} & \mathrm{Def}(R^2\varphi_*\mathcal{K}^M_{2,\mathbb{P}(\mathcal{E})})(R)\arrow{r}\arrow{d} &\mathrm{Def}(\mathrm{Pic}_{X/k,(Zar)})(R)\arrow{r}\arrow{d} & 0  \\
				0\arrow{r} &  \HH^2(X,\Omega_{X}^1)\otimes_k \mathfrak{m}_R\arrow{r} & \HH^2(\mathbb{P}(\mathcal{E}),\Omega_{\mathbb{P}(\mathcal{E})}^1)\otimes_k \mathfrak{m}_R\arrow{r} & \HH^1(X,\mathcal{O}_X)\otimes_k \mathfrak{m}_R \arrow{r} & 0.
		\end{tikzcd}\end{equation}
	Here the rightmost vertical arrow is an isomorphism by Grothendieck's existence theorem, cf.\ \cite[\href{https://stacks.math.columbia.edu/tag/089N}{Tag 089N}]{stacks-project}. Therefore, if either the left or the middle vertical arrow is a surjection, then the other is as well. By varying the power series ring $R$, and by using Remark \ref{rmk:changingrings}, it follows that $R^2\pi_*\mathcal{K}^M_{2,X}$ is effectively pro-representable if and only if $R^2\varphi_*\mathcal{K}^M_{2,\mathbb{P}(\mathcal{E})}$ is effectively pro-representable. The analogous theorem for the \'etale sheafifications $(R^2\pi_*\mathcal{K}^M_{2,X})_{\acute{e}t}$ and $(R^2\varphi_*\mathcal{K}^M_{2,\mathbb{P}(\mathcal{E})})_{\acute{e}t}$ is proved similarly noting that, in each of the above diagrams, all splittings descend to Galois invariants.
\end{proof}

\begin{lem}\label{lem: blowup}
	Fix an algebraic extension $k/\mathbb{Q}$ and let $\pi:X\rightarrow k$  be a smooth, proper, and geometrically connected scheme. Let $Z\subset X$ be a smooth subscheme of $X$ and let $\varphi:\mathrm{Bl}_Z(X)\rightarrow k$ be the structure map of the blow-up of $X$ along $Z$. Then $R^2\pi_{*}\mathcal{K}^M_{2,X}$ is effectively pro-representable if and only if $R^2\varphi_* \mathcal{K}^M_{2,\mathrm{Bl}_Z(X)}$ effectively pro-representable.
	
	The above statement also holds replacing all Zariski sheaves with their \'etale sheafifications. 
\end{lem}

\begin{proof}
It suffices to assume that $Z$ is smooth and connected with $\mathrm{codim}(Z,X)\geq 2$. As before, let $R=k[[t_1,...,t_r]]$ be a power series ring in finitely many variables with maximal ideal $\mathfrak{m}_R$. Set $S=\mathrm{Spec}(R)$. We write $\mathcal{N}_{Z/X}$ for the normal sheaf on $Z$ for the inclusion $Z\subset X$ and we write $\Theta:\mathbb{P}(\mathcal{N}_{Z/X})\rightarrow X$ for the associated projective bundle map. There's an isomorphism $\mathrm{Pic}(\mathbb{P}(\mathcal{N}_{Z/X}))\cong \Theta^*\mathrm{Pic}(Z)\oplus \mathbb{Z}c_1(\mathcal{O}_{\mathbb{P}(\mathcal{N}_{Z/X})}(1))$. By the blow-up formula for regular embeddings \cite[Proposition 6.7 (e)]{MR1644323} there is an exact sequence \[0\rightarrow \CH^2(X)\rightarrow \CH^2(\mathrm{Bl}_Z(X))\rightarrow \mathrm{Pic}(\mathbb{P}(\mathcal{N}_{Z/X}))/A\rightarrow 0\] where $A=0$ if $\mathrm{codim}(Z,X)>2$ or, if $\mathrm{codim}(Z,X)=2$, then $A$ is the infinite cyclic subgroup generated by $\Theta^*c_1(\mathcal{N}_{Z/X})+c_1(\mathcal{O}_{\mathbb{P}(\mathcal{N}_{Z/X})}(1))$. Since Remark \ref{rmk: gerstform} applies to the triple $(X_S, Z_S, \mathrm{Bl}_{Z_S}(X_S))$ and, due to \cite[\href{https://stacks.math.columbia.edu/tag/0E9J}{Tag 0E9J}]{stacks-project}, the blow-up formula also provides an exact sequence \[0\rightarrow \CH^2(X_S)\rightarrow \CH^2(\mathrm{Bl}_{Z_S}(X_S))\rightarrow \mathrm{Pic}(\mathbb{P}(\mathcal{N}_{Z_S/X_S}))/A_S\rightarrow 0\] with $A_S$ characterized by the codimension of $Z\subset X$ similarly.
	
If $\mathrm{codim}(Z,X)=2$, then pulling back induces isomorphisms \[ \mathrm{Pic}(Z)\cong \mathrm{Pic}(\mathbb{P}(\mathcal{N}_{Z/X}))/A \quad \mbox{and}\quad \mathrm{Pic}(Z_S)\cong  \mathrm{Pic}(\mathbb{P}(\mathcal{N}_{Z_S/X_S}))/A_S.\] Now, regardless of the codimension of $Z$ in $X$, the above sequences on Chow groups produce an exact sequence 
	\begin{equation}\begin{tikzcd}[column sep = small] 0\arrow{r} & \mathrm{Def}(R^2\pi_*\mathcal{K}^M_{2,X})(R)\arrow{r} & \mathrm{Def}(R^2\varphi_*\mathcal{K}^M_{2,\mathrm{Bl}_Z(X)})(R)\arrow{r} &\mathrm{Def}(\mathrm{Pic}_{Z/k,(Zar)})(R)\arrow{r} & 0.
		\end{tikzcd}\end{equation} If either of $R^2\pi_*\mathcal{K}^M_{2,X}$ or $R^2\varphi_*\mathcal{K}^M_{2,\mathrm{Bl}_Z(X)}$ are pro-representable, then pulling back along the map $\mathrm{Bl}_Z(X)\rightarrow X$ induces an exact commutative ladder
	\begin{equation}\begin{tikzcd}[column sep = small] 0\arrow{r} & \mathrm{Def}(R^2\pi_*\mathcal{K}^M_{2,X})(R)\arrow{r}\arrow{d} & \mathrm{Def}(R^2\varphi_*\mathcal{K}^M_{2,\mathrm{Bl}_Z(X)})(R)\arrow{r}\arrow{d} &\mathrm{Def}(\mathrm{Pic}_{Z/k,(Zar)})(R)\arrow{r}\arrow{d} & 0\\
				0\arrow{r} & \HH^2(X,\Omega^1_X)\otimes_k \mathfrak{m}_R\arrow{r} & \HH^2(\mathrm{Bl}_Z(X),\Omega^1_{\mathrm{Bl}_Z(X)})\otimes_k \mathfrak{m}_R\arrow{r} & \HH^1(Z,\mathcal{O}_Z)\otimes_k \mathfrak{m}_R\arrow{r} & 0.
		\end{tikzcd}\end{equation}
	The rightmost vertical arrow is an isomorphism. Hence, if either of the left or the middle vertical arrows were surjective, then the other would be as well. This allows us to conclude as before.
\end{proof}

	\begin{thm}\label{thm: birational}
	Fix an algebraic extension $k/\mathbb{Q}$ and let $\pi_X:X\rightarrow k$ and $\pi_Y:Y\rightarrow k$ be two smooth, proper, and geometrically connected $k$-schemes. Suppose that $X$ and $Y$ are stably birational over $k$.
	
	Then $(R^2\pi_{X*}\mathcal{K}^M_{2,X})_\tau$ is algebraizable if and only if $(R^2\pi_{Y*}\mathcal{K}^M_{2,Y})_\tau$ algebraizable for either $\tau=\mathrm{Zar}, \acute{e}t$.
\end{thm}

\begin{proof}
	This follows immediately from the content of Section \ref{sec:back}, Lemma \ref{lem: probunlem}, Lemma \ref{lem: blowup}, and the Weak Factorization theorem over $k$, \cite[Theorem 0.0.1 (1)]{MR2483958}. Namely, suppose $(R^2\pi_{X*}\mathcal{K}^M_{2,X})_\tau$ is algebraizable and let $\varphi_X:X\times \mathbb{P}^r\rightarrow k$ and $\varphi_Y:Y\times \mathbb{P}^s\rightarrow k$ be birationally equivalent $k$-schemes for some $r,s\geq 0$. 
	
	Then the higher push forward functor $(R^2\varphi_{X*}\mathcal{K}^M_{2,X\times \mathbb{P}^r})_\tau$ is effectively pro-representable by of Lemma \ref{lem: probunlem} and so algebraizable. Any birational equivalence between the two schemes $X\times \mathbb{P}^r$ and $Y\times \mathbb{P}^s$ can be factored into a sequence of blow-ups and blow-downs at smooth centers, by the Weak Factorization theorem, so that this implies $(R^2\varphi_{Y*}\mathcal{K}^M_{2,Y\times \mathbb{P}^s})_\tau$ is then effectively pro-representable by Lemma \ref{lem: blowup} and hence algebraizable. We can then conclude that $(R^2\pi_{Y*}\mathcal{K}^M_{2,Y})_\tau$ is effectively pro-representable, and hence algebraizable, by use of Lemma \ref{lem: probunlem} again.
\end{proof}

\begin{exmp}
Let $k$ be an algebraic extension of $\mathbb{Q}$ and suppose that $\pi:X\rightarrow k$ is a smooth, proper, and geometrically connected surface with geometric genus $p_g(X)=\dim\HH^0(X,\Omega^2_X)=0$. Then both $R^2\pi_*\mathcal{K}^M_{2,X}$ and $(R^2\pi_*\mathcal{K}^M_{2,X})_{\acute{e}t}$ are pro-representable by Bloch's theorem \cite{MR371891}. If, moreover, $\dim\HH^1(X,\mathcal{O}_X)\leq 1$ then this pro-representability is effective. To see this, one notes $\HH^1(X,\mathcal{O}_X)\cong \HH^2(X,\Omega_X^1)$. If the latter vanishes, then effectivity is trivial. Otherwise, it follows from \cite[Theorem 1.1]{MR4112683}.

By Theorem \ref{thm: birational}, any smooth and proper variety $\phi:Y\rightarrow k$ stably birational to such an $X$ therefore has algebraizable $R^2\phi_*\mathcal{K}^M_{2,Y}$ and $(R^2\phi_*\mathcal{K}^M_{2,Y})_{\acute{e}t}$.
\end{exmp}

\section{In the case $i=1$}
The proof that $R^1\pi_*\mathcal{K}^M_{2,X}$ is a stable birational invariant is similar to the case of $R^2\pi_*\mathcal{K}^M_{2,X}$, but relies on a relative version of the $K$-cohomology groups of Rost \cite{MR1418952}. We include all of the necessary set-up for these groups, along with the results about them that we will use here, in an appendix below.

\begin{lem}\label{lem: newrep}
	Let $k$ be a field and let $\pi:X\rightarrow k$ be a proper scheme. Set $A=\HH^0(X,\mathcal{O}_X)$ and $S=\mathrm{Spec}(A)$. Then the functor $\pi_*\mathcal{K}^M_{1,X}$ is representable by the group $k$-scheme Weil restriction $\mathrm{Res}_{S/k}(\mathbb{G}_{m,S})$.
\end{lem}

\begin{proof}
	We first remark that, since $A$ is a finite $k$-algebra as $X$ is proper \cite[\href{https://stacks.math.columbia.edu/tag/02O6}{Tag 02O6}]{stacks-project}, the Weil restriction $\mathrm{Res}_{S/k}(\mathbb{G}_{m,S})$ exists as a $k$-scheme, see \cite[7.6, Theorem 4]{MR1045822}. Now there exists a canonically defined factorization of $\pi$ into a composition \[X\xrightarrow{\phi} S \xrightarrow{\rho} k\] and we use this factorization to define a natural transformation \[ \mathrm{Res}_{S/k}(\mathbb{G}_{m,S})\rightarrow \pi_*\mathbb{G}_{m,X}=\pi_*\mathcal{K}_{1,X}^M.\] 
	
	For any scheme $T/k$ one has natural identifications \[\pi_*\mathbb{G}_{m,X}(T)=\mathrm{Hom}_X(T_X,\mathbb{G}_{m,X})=\mathcal{O}_{T_X}(T_X)^\times,\] where $\mathcal{O}_{T_X}(T_X)^\times$ is the group of units of the ring $\mathcal{O}_{T_X}(T_X),$ and
	\[\mathrm{Hom}_k(T,\mathrm{Res}_{S/k}(\mathbb{G}_{m,S}))=\mathrm{Hom}_S(T_S,\mathbb{G}_{m,S})=\mathcal{O}_{T_S}(T_S)^\times.\] We take as our definition of a natural transformation the map \[\pi_*\mathbb{G}_{m,X}(T)\rightarrow \mathrm{Hom}_k(T,\mathrm{Res}_{S/k}(\mathbb{G}_{m,S}))\] coming from taking global sections of the induced map $\mathcal{O}_{T_S}\rightarrow \phi_{T*}\mathcal{O}_{T_X}$. But note that since $X\rightarrow S$ is proper, and $T_S\rightarrow S$ is flat, the morphism $\mathcal{O}_{T_S}\rightarrow \phi_{T*}\mathcal{O}_{T_X}$ is an isomorphism \cite[\href{https://stacks.math.columbia.edu/tag/02KH}{Tag 02KH}]{stacks-project}.
\end{proof}

\begin{lem}\label{lem: probunlem2}
	Fix an algebraic extension $k/\mathbb{Q}$ and let $\pi:X\rightarrow k$  be a smooth, proper, and geometrically connected scheme. Let $\mathcal{E}$ be a finite rank locally free sheaf on $X$ and $\varphi:\mathbb{P}(\mathcal{E})\rightarrow k$ be the structure map of the associated projective bundle. Then $R^1\pi_*\mathcal{K}^M_{2,X}$ is effectively pro-representable if and only if $R^1\varphi_* \mathcal{K}^M_{2,\mathbb{P}(\mathcal{E})}$ is effectively pro-representable.
	
	The above statement also holds replacing all Zariski sheaves with their \'etale sheafifications.
\end{lem}

\begin{proof}
We write $S=\mathrm{Spec}(R)$ for a power series ring $R=k[[t_1,...,t_r]]$. Then the proof is exactly the same as that of Lemma \ref{lem: probunlem} with the following changes. First, with the notation as in the appendix, there are isomorphisms (cf.\ Corollary \ref{cor: gerst}) \[\HH^1(X_S,\mathcal{K}^M_{2,X_S})\cong A^1(X_S;K_2^M) \quad \mbox{and} \quad \HH^1(\mathbb{P}(\mathcal{E})_S,\mathcal{K}^M_{2,{\mathbb{P}(\mathcal{E})_S}})\cong A^1(\mathbb{P}(\mathcal{E})_S;K_2^M)\] with the given $K$-cohomology groups. Since the lemma is trivial if $\mathrm{rk}(\mathcal{E})=1$, assume $\mathrm{rk}(\mathcal{E})>1$. Then the pullback along the projection $\mathbb{P}(\mathcal{E})_S\rightarrow X_S$ induces a short exact sequence \[0\rightarrow A^1(X_S;K_2^M)\rightarrow A^1(\mathbb{P}(\mathcal{E})_S;K_2^M) \rightarrow  A^0(X_S;K_1^M)\rightarrow 0\] by the projective bundle formula. Similar to before there is then a commutative diagram \begin{equation}\begin{tikzcd}[column sep = small] 0\arrow{r} & \mathrm{Def}(R^1\pi_*\mathcal{K}^M_{2,X})(R)\arrow{r}\arrow{d} & \mathrm{Def}(R^1\varphi_*\mathcal{K}^M_{2,\mathbb{P}(\mathcal{E})})(R)\arrow{r}\arrow{d} &\mathrm{Def}(\pi_*\mathcal{K}^M_{1,X})(R)\arrow{r}\arrow{d} & 0  \\
		0\arrow{r} &  \HH^1(X,\Omega_{X}^1)\otimes_k \mathfrak{m}_R\arrow{r} & \HH^1(\mathbb{P}(\mathcal{E}),\Omega_{\mathbb{P}(\mathcal{E})}^1)\otimes_k \mathfrak{m}_R\arrow{r} & \HH^0(X,\mathcal{O}_X)\otimes_k \mathfrak{m}_R \arrow{r} & 0.
\end{tikzcd}\end{equation} Since $\pi_*\mathcal{K}_{1,X}^M$ is representable by $\mathbb{G}_{m,k}$ by Lemma \ref{lem: newrep}, the rightmost vertical arrow is an isomorphism, and the lemma follows as before.
\end{proof}

\begin{lem}
Fix an algebraic extension $k/\mathbb{Q}$ and let $\pi:X\rightarrow k$  be a smooth, proper, and geometrically connected scheme. Let $Z\subset X$ be a smooth subscheme of $X$ and let $\varphi:\mathrm{Bl}_Z(X)\rightarrow k$ be the structure map of the blow-up of $X$ along $Z$. Then $R^1\pi_{*}\mathcal{K}^M_{2,X}$ is effectively pro-representable if and only if $R^1\varphi_* \mathcal{K}^M_{2,\mathrm{Bl}_Z(X)}$ effectively pro-representable.

The above statement also holds replacing all Zariski sheaves with their \'etale sheafifications. 
\end{lem}

\begin{proof}
Write $S=\mathrm{Spec}(R)$ for a power series ring $R=k[[t_1,...,t_r]]$ for some integer $r>0$. We can assume $Z$ is smooth and connected with $\mathrm{codim}(Z,X)\geq 2$. Note we have \[\HH^1(X_S,\mathcal{K}^M_{2,X_S})\cong A^1(X_S;K_2^M) \quad \mbox{and} \quad \HH^1(\mathrm{Bl}_{Z_S}(X_S),\mathcal{K}^M_{2,{\mathrm{Bl}_{Z_S}(X_S)}})\cong A^1(\mathrm{Bl}_{Z_S}(X_S);K_2^M).\] By the blow-up formula for relative $K$-homology, see Remark \ref{rmk: cohho} and Proposition \ref{prop: blowup} below, there's an exact sequence \[ 0\rightarrow A^1(X_S;K_2^M)\rightarrow A^1(\mathrm{Bl}_{Z_S}(X_S);K^M_2)\rightarrow A^0(Z_S;K_1^M)\rightarrow 0.\] Similar to before, there is then a commutative diagram \begin{equation}\begin{tikzcd}[column sep = small] 0\arrow{r} & \mathrm{Def}(R^1\pi_*\mathcal{K}^M_{2,X})(R)\arrow{r}\arrow{d} & \mathrm{Def}(R^1\varphi_*\mathcal{K}^M_{2,\mathrm{Bl}_Z(X)})(R)\arrow{r}\arrow{d} &\mathrm{Def}(\rho_*\mathcal{K}^M_{1,Z})(R)\arrow{r}\arrow{d} & 0  \\
		0\arrow{r} &  \HH^1(X,\Omega_{X}^1)\otimes_k \mathfrak{m}_R\arrow{r} & \HH^1(\mathrm{Bl}_Z(X),\Omega_{\mathrm{Bl}_Z(X)}^1)\otimes_k \mathfrak{m}_R\arrow{r} & \HH^0(Z,\mathcal{O}_Z)\otimes_k \mathfrak{m}_R \arrow{r} & 0
\end{tikzcd}\end{equation} where $\rho:Z\rightarrow k$ is the structure morphism of $Z$. Since $\rho_*\mathcal{K}_{1,Z}^M$ is representable by $\mathbb{G}_{m,k}$ by Lemma \ref{lem: newrep}, the rightmost vertical arrow is an isomorphism, and the lemma follows as before.
\end{proof}
	
As a result of the above two lemmas, we get:

\begin{thm}\label{thm: birational2}
	Fix an algebraic extension $k/\mathbb{Q}$ and let $\pi_X:X\rightarrow k$ and $\pi_Y:Y\rightarrow k$ be two smooth, proper, and geometrically connected $k$-schemes. Suppose that $X$ and $Y$ are stably birational over $k$.
	
	Then $(R^1\pi_{X*}\mathcal{K}^M_{2,X})_\tau$ is algebraizable if and only if $(R^1\pi_{Y*}\mathcal{K}^M_{2,Y})_\tau$ algebraizable for either $\tau=\mathrm{Zar}, \acute{e}t$.$\hfill\square$
\end{thm}

\begin{exmp}
In \cite[Theorem C]{MR0688259} it's shown that if $X$ is a geometrically rational surface over $k$ with a $k$-rational point, then the group $\HH^1(X,\mathcal{K}_{2,X}^M)$ is isomorphic with the group of $k$-points of the $k$-torus that's dual to the $\mathrm{Gal}(\overline{\mathbb{Q}}/k)$-module $\mathrm{Pic}(X_{\overline{\mathbb{Q}}})$.
	
It follows from Theorem \ref{thm: birational2} that if $\pi:X\rightarrow k$ is a rational surface over an algebraic field extension $k/\mathbb{Q}$, then $R^1\pi_*\mathcal{K}^M_{2,X}$ and $(R^1\pi_*\mathcal{K}^M_{2,X})_{\acute{e}t}$ are also algebraizable since $X$ is stably birational to $\mathrm{Spec}(k)$.
\end{exmp}

\noindent\textbf{Acknowledgments}. I'd like to thank Niranjan Ramachandran for both bringing my attention to the problem of representability of higher $K$-cohomology groups which started this line of inquiry and for the many conversations that we've had about this work during its genesis.

\appendix

\section{$K$-cohomology and relative $K$-homology}
In Rost's construction of $K$-(co)homology groups \cite{MR1418952}, there is an underlying assumption that all schemes considered are of finite type over a fixed perfect ground field. However, these assumptions aren't needed for many of the constructions (see Remark 2.8, ibid.) and, in this paper, we want to use many of the properties of $K$-cohomology groups for much more general schemes (e.g.\ those of Remark \ref{rmk: gerstform}).
	
The first objective for this appendix, then, is to set-up a generalized version of both $K$-cohomology groups, and of relative $K$-homology groups, with the intent of applying them to the schemes of Remark \ref{rmk: gerstform}.
	
One generalization of Rost's constructions is given in the reference \cite[Chapter IX]{MR2427530} which, ultimately, gives a very satisfactory theory for schemes of finite type over a field. At some points though, technical aspects of dimension make the constructions of \cite[Chapter IX]{MR2427530} less suitable for the generalized setting we're interested in (cf.\ Remark \ref{rmk: a1} and Remark \ref{rmk: flatreldim}).
	
It turns out that these subtleties can be remedied (for the most part) with only some slight modifications. So, we use the reference \cite[Chapter IX]{MR2427530} as a guide for many statements and their proofs, pointing out what changes should be made to where, in order to get a working theory in the setting we're interested in.

\subsection{Topological preliminaries}
Let $X$ be a sober Noetherian topological space of finite dimension (e.g.\ the underlying topological space of a finite dimensional Noetherian scheme). We will use the terminology below, following \cite{MR3631826}.

\begin{defn} Let $X$ be as above.
\begin{enumerate}
\item $X$ is said to be \textit{equidimensional} if all irreducible components of $X$ have the same dimension.
\item $X$ is said to be \textit{equicodimensional} if all minimal irreducible closed subsets of $X$ have the same codimension in $X$.
\item $X$ is \textit{catenary} if for all irreducible closed subsets $T\subset T'$, every maximal chain of irreducible closed subsets starting with $T$ and ending with $T'$ has the same length.
\item $X$ is \textit{weakly biequidimensional} if $X$ is equidimensional, equicodimensional, and catenary.
\item $X$ is \textit{biequidimensional} if all maximal chains of irreducible closed subsets of $X$ have the same length.
\end{enumerate}
\end{defn}

A biequidimensional space is weakly biequidimensional, but not conversely \cite[\S2-3]{MR3631826}. However, a space which is equidimensional, catenary, and such that all of its irreducible components are equicodimensional is biequidimensional.

\begin{defn}
Let $f:X\rightarrow Y$ be a continuous map of topological spaces. We say that $f$ is \textit{catenarious} if \[\mathrm{codim}(T,T')\geq \mathrm{codim}(\overline{f(T)},\overline{f(T')})\] for all of irreducible closed subsets $T\subset T'$ of $X$.
\end{defn}

\begin{rmk}
Suppose that $X$ and $Y$ are catenary. Then for a continuous map $f:X\rightarrow Y$ to be catenarious, it's necessary and sufficient that for any irreducible closed subsets $T\subset T'$ with $\mathrm{codim}(T,T')=1$ we have $\mathrm{codim}(\overline{f(T)},\overline{f(T')})=0$ or $\mathrm{codim}(\overline{f(T)},\overline{f(T')})=1$. Necessity is clear. To see that this condition is sufficient, let $T\subset T''$ be any inclusion of irreducible closed subsets of $X$ with $\mathrm{codim}(T,T'')=n$. Then one can find a maximal chain of irreducible closed subsets \[ T=T_0\subset T_1\subset \cdots \subset T_n=T''.\] As $X$ and $Y$ are catenary, we have \begin{align*} \mathrm{codim}(\overline{f(T)},\overline{f(T'')})&=\sum_{i=0}^{n-1}\mathrm{codim}(\overline{f(T_i)},\overline{f(T_{i+1})})\\ &\leq \sum_{i=0}^{n-1}\mathrm{codim}(T_i,T_{i+1})\\ &=\mathrm{codim}(T,T'').\end{align*}
\end{rmk}

\begin{exmp}
Let $X$ be a finite dimensional Noetherian scheme. Let $Y=\mathrm{Spec}(R)$ be the spectrum of a discrete valuation ring $R$. Then any morphism $f:X\rightarrow Y$ is catenarious.
\end{exmp}

Roughly speaking, catenarious maps are a maps which respect the codimension relation in a strong sense. The following lemma gives a sufficient condition to guarantee that a map is catenarious.

\begin{lem}\label{lem: fincat}
	Let $X$ and $Y$ be catenary Noetherian schemes of finite dimension. Suppose that $f:X\rightarrow Y$ is an integral morphism of schemes and that $X$ has equicodimensional irreducible components. Then $f$ is catenarious.
\end{lem}

\begin{proof}
	Let $T\subset T'$ be two irreducible closed subsets of $X$. To verify that $f$ is catenarious we suppose that $\mathrm{codim}(T,T')< \mathrm{codim}(f(T),f(T'))$ and aim to deduce a contradiction. Replacing $X$ by $T'$, we may assume that $X$ is irreducible. Similarly, replacing $Y$ by $f(X)$, we can assume that $f$ is integral and surjective.
	
	We can find a maximal chain of irreducible closed subsets $T_0=T\subset \cdots\subset T_n\subset X$ and a maximal chain of irreducible closed subsets $T_{-m}\subset \cdots \subset T_{-1}\subset T_0$ giving a maximal chain $\mathcal{C}$ of $X$ containing $T$: \[\mathcal{C}:\quad T_{-m}\subset \cdots\subset T_0=T\subset\cdots \subset T_n\subset X.\] Since $f$ is integral, the image $f(\mathcal{C})$ of this chain in $Y$ is a chain of irreducible closed subsets containing $f(T)$. If $\mathrm{codim}(T,T')<\mathrm{codim}(f(T),f(T'))$ then the chain $f(\mathcal{C})$ is not maximal, and we can extend it to a maximal chain of strictly longer length, say $\mathcal{D}$ (here we use that $Y$ is catenary). 
	
	Now since $f$ lifts specializations \cite[\href{https://stacks.math.columbia.edu/tag/0066}{Tag 0066}]{stacks-project}, we can find a chain $\mathcal{C}'$ of irreducible closed subsets of $X$ so that $f(\mathcal{C}')=\mathcal{D}$. Since $X$ is catenary, this implies that $X$ has two minimal irreducible closed subsets with differing codimension in $X$, contradicting the assumption that $X$ is equicodimensional.
\end{proof}

\begin{exmp}\label{exmp: normcat}
	Suppose that $X$ is an integral, equicodimensional, excellent Noetherian scheme of finite dimension. Let $f:X^\nu\rightarrow X$ be the normalization of $X$ in its field of fractions. Then $X^\nu$ is equicodimensional and $f$ is finite as we now show. Thus Lemma \ref{lem: fincat} shows that $f$ is catenarious.
	
	Since $X$ is excellent, the normalization map is finite by \cite[\href{https://stacks.math.columbia.edu/tag/035R}{Tag 035R}]{stacks-project}. Since $f$ is surjective, we have $\mathrm{dim}(X)=\mathrm{dim}(X^\nu)$ by \cite[\href{https://stacks.math.columbia.edu/tag/0ECG}{Tag 0ECG}]{stacks-project}. Suppose that $x'\in X^\nu$ is a minimal irreducible closed subset, i.e.\ a closed point. Then $x=f(x')$ is a closed point of $X$. Since $X$ is equicodimensional and irreducible, there is an equality $\mathrm{codim}(x,X)=\mathrm{dim}(X)$. Let $\mathcal{D}$ be a chain of irreducible closed subsets of $X$ containing $x$ with this maximal length.
	
	Since the going down theorem holds for the normalization of a domain, it follows that generalizations lift along the map $f$. In particular, there is a maximal chain $\mathcal{C}$ of irreducible closed subsets of $X^\nu$ containing $x'$ such that $f(\mathcal{C})=\mathcal{D}$. This shows that $\mathrm{codim}(x',X^\nu)\geq \mathrm{dim}(X)=\mathrm{dim}(X^\nu)$ and so equality must hold. Since $x'$ was arbitrary, it follows that $X^\nu$ is equicodimensional as well.
\end{exmp}

In general, however, even very reasonable maps will not be catenarious. For this reason, it's advantageous to introduce the following weaker notion:

\begin{defn}
	Let $f:X\rightarrow Y$ be a continuous map of topological spaces. We say that $f$ is \textit{weakly catenarious} if for every point $y\in Y$ and for every irreducible component $T'\subset f^{-1}(y)$ we have \[\mathrm{codim}(T,\overline{T'})\geq \mathrm{codim}(\overline{f(T)},\overline{f(T')})\] for every irreducible closed subset $T\subset \overline{T'}\subset X$ with $\mathrm{codim}(T,\overline{T'})=1$.
\end{defn}

The following lemma gives a sufficient condition to guarantee a map is weakly catenarious.

\begin{lem}\label{lem: flatcat}
	Let $X$ be a locally Noetherian catenary scheme. Suppose that $f:X\rightarrow Y$ is a flat morphism of schemes and assume also that, for every point $y\in Y$ and for every point $x\in \overline{X_y}=\overline{f^{-1}(y)}$, the scheme $\mathrm{Spec}(\mathcal{O}_{\overline{X_y},x})$ is equidimensional. Then $f$ is weakly catenarious.	
\end{lem}

\begin{proof}
Let $y\in Y$ be a point and let $T'$ be an irreducible component of $f^{-1}(y)$. Let $T$ be any irreducible closed subset of $\overline{T'}\subset X$ with $\mathrm{codim}(T,\overline{T'})=1$. Let $x\in X$ be the generic point of $T$ and $y'=f(x)$.

Since $f$ is flat, the inclusion $\overline{f^{-1}(y)}\subset f^{-1}(\overline{\{y\}})$ is an equality. This is simply because if $z\in f^{-1}(\overline{\{y\}})$ then either $f(z)=y$ or $y$ specializes to $f(z)$. In the latter case, there is a point $z'$ specializing to $z$ and mapping to $y$ by \cite[\href{https://stacks.math.columbia.edu/tag/03HV}{Tag 03HV}]{stacks-project}. Hence $z\in \overline{\{z'\}}\subset \overline{f^{-1}(y)}$ since closed sets are stable under specialization.

Replacing $Y$ with $\overline{f(T')}=\overline{\{y\}}$ and $X$ with $\overline{f^{-1}(y)}=f^{-1}(\overline{\{y\}})$, we may assume that $f$ is flat and dominant. Further, we may replace $X$ with $\mathrm{Spec}(\mathcal{O}_{X,x})$ and $Y$ with $\mathrm{Spec}(\mathcal{O}_{Y,y'})$.

The scheme $\overline{T'}$ is now an irreducible component of $X$. Further. the scheme $T=x$ is a minimal irreducible closed subset of $X$ contained in every maximal chain of irreducible closed subsets of $X$. Since $X$ is assumed equidimensional and catenary, and because codimension $\mathrm{codim}(T,\overline{T'})=1$, it follows that every maximal chain of irreducible closed subsets of $X$ must have length $1$. 

Now if $\mathrm{codim}(\overline{f(T)},\overline{f(T')})>1$, then there is a chain of irreducible closed subsets $\mathcal{D}$ of $Y$ of length more than $1$. Since generalizations lift under flat maps, if this inequality held, then we would be able to lift $\mathcal{D}$ to a chain of length more than $1$ in $X$ by \cite[\href{https://stacks.math.columbia.edu/tag/03HV}{Tag 03HV}]{stacks-project}. Hence $\mathrm{codim}(\overline{f(T)},\overline{f(T')})\leq 1$ as desired. 
\end{proof}

\begin{exmp}
Let $X$ and $Y$ be catenary Noetherian schemes of finite dimension. Suppose that $f:X\rightarrow Y$ is flat and assume also that for every point $y\in Y$, the fiber $X_y$ is irreducible. Then for every $x\in \overline{X_y}$ the scheme $\mathrm{Spec}(\mathcal{O}_{\overline{X_y},x})$ is equidimensional. Lemma \ref{lem: flatcat} shows that such a map $f$ is weakly catenarious.
\end{exmp}

\begin{exmp}\label{exmp: localweakcat}
Let $(R,\mathfrak{m}_R)$ and $(S,\mathfrak{m}_S)$ be two local Noetherian rings with Krull dimension $\mathrm{dim}(\mathrm{Spec}(S))=\mathrm{dim}(\mathrm{Spec}(R))=2$. Suppose that $\mathrm{Spec}(R)$ is irreducible and suppose that $\mathrm{Spec}(S)$ is equidimensional. Let \[f:\mathrm{Spec}(S)\rightarrow \mathrm{Spec}(R)\] be any flat map such that $f^{-1}(\mathfrak{m}_R)=\mathfrak{m}_S$. Then $f$ is weakly catenarious as we now show.

Let $\eta_R\subset \mathrm{Spec}(R)$ be the generic point. Then $f^{-1}(\eta_R)$ contains all of the generic points of $\mathrm{Spec}(S)$ since $f$ is flat. So the closure of an irreducible component of $f^{-1}(\eta_R)$ is an irreducible component $T'$ of $\mathrm{Spec}(S)$. Let $t\in \mathrm{Spec}(S)$ be such that $\mathrm{codim}(\overline{\{t\}},T')=1$. Since $S$ is equidimensional, we have that $\mathrm{dim}(T')=2$. Then $f(t)\neq \mathfrak{m}_S$, so $\mathrm{codim}(\overline{\{f(t)\}}, \mathrm{Spec}(R))\leq 1$ as desired.

Next, if $z\in \mathrm{Spec}(R)$ is any point with $\mathrm{codim}(\overline{\{z\}}, \mathrm{Spec}(R))=1$ then the closure $\overline{T'}$ of any irreducible component of $T'\subset f^{-1}(z)$ satisfies $\mathrm{codim}(\overline{T'},\mathrm{Spec}(S))=1$. The only irreducible closed subset of $\overline{T'}$ is $\mathfrak{m}_S$, and the claim follows.
\end{exmp}

\subsection{$K$-cohomology} Let $X$ be a separated and excellent scheme. We set \[C(X)=\bigoplus_{x\in X} \bigoplus_{n\in \mathbb{Z}} K^M_n(\kappa(x))\] where $K^M_n(\kappa(x))$ is the $n$th Milnor $K$-theory of the residue field of $x$. We construct an endomorphism $d_X$ of the group $C(X)$ as follows. It suffices to specify the map $d_X$ componentwise as \[(d_X)_{x'}^x:K^M_n(\kappa(x))\rightarrow K^M_m(\kappa(x'))\] for every pair of points $x,x'\in X$ and for all pairs of integers $m,n\in \mathbb{Z}$. We set this map to be $0$ unless $m=n-1$ and $x'$ is a specialization of $x$ such that $\mathrm{codim}(\overline{\{x'\}},\overline{\{x\}})=1$. Otherwise, we set \[(d_X)_{x'}^x=\sum_{x_i'} c_{\kappa(x_i')/\kappa(x')}\circ\partial_{v_i} :K^M_n(\kappa(x))\rightarrow K^M_{n-1}(\kappa(x'))\] to be sum of the composition of the residue maps \[\partial_{v_i}:K^M_n(\kappa(x))\rightarrow K^M_{n-1}(\kappa(x_i'))\] and norm maps \[c_{\kappa(x_i')/\kappa(x')}:K^M_{n-1}(\kappa(x_i'))\rightarrow K^M_{n-1}(\kappa(x'))\] over all valuations $v_i$ determined by points $x_i'$ lying over $x'$ in the normalization of $\mathcal{O}_{\overline{\{x\}},x'}$. Note that for any $f\in C(X)$ we have $d_X(f)\in C(X)$ by \cite[Lemma 49.1]{MR2427530}.

\begin{rmk}\label{rmk: a1}
	The definition that we give for the endomorphism $d_X$ is strictly different from the definition of $d_X$ given in \cite[\S49.A]{MR2427530}, which specifies that the $(x,x')$ component of $d_X$ is possibly nonzero only for $x,x'\in X$ such that $\mathrm{dim}(\overline{\{x\}})=\mathrm{dim}(\overline{\{x'\}})+1$.
	
	For instance, let $T=\mathrm{Spec}(k[[x]][y])$ and consider the maximal ideal $\mathfrak{m}=(xy-1)$. Then $\mathrm{dim}(T)=2$ and $\mathrm{dim}(\mathfrak{m})=0$, but $\mathrm{codim}(\mathfrak{m},T)=1$. So $\mathfrak{m}$ is a point considered in our definition which is not considered in the definition of the map $d_X$ in \cite[\S49.A]{MR2427530}.
	
	We will frequently replace conditions on dimension that are used in \cite[\S49.A]{MR2427530} by either conditions on codimension or by conditions on relative dimension (as defined in \cite[\S20.1]{MR1644323}). The latter two of these three seem to have better properties than the former.
\end{rmk}

Let $Y$ be another separated and excellent scheme. Associated to any proper morphism $f:X\rightarrow Y$ is a functorial pushforward \[f_*:C(X)\rightarrow C(Y)\] defined as in \cite[\S49.C]{MR2427530}. For $x\in X$ and $y\in Y$ the map $f_*$ is defined on componentwise \[(f_*)^x_y:K_n^M(\kappa(x))\rightarrow K_m^M(\kappa(y))\] to be trivial unless $y=f(x)$, the extension $\kappa(x)/\kappa(y)$ is finite, and $n=m$. If these conditions are met, then one sets $(f_*)^x_y=c_{\kappa(x)/\kappa(y)}$ to be the associated norm map.

\begin{lem}\label{lem: properpush}
Let $X$ and $Y$ be separated and excellent schemes. Let $f:X\rightarrow Y$ be a proper morphism of schemes. Then the proper pushforward commutes with the endomorphisms $d_X$ and $d_Y$, i.e.\ $d_Y\circ f_*=f_*\circ d_X$.
\end{lem}

\begin{proof}
Let $x\in X$ and $y'\in Y$ be two points. Set $y=f(x)$. If $y'\notin\overline{\{y\}}$ then we have that $(d_Y\circ f_*)_{y'}^x=0$ and $(f_*\circ d_X)_{y'}^x=0$. This proves the claim in this case.

We can thus assume $y'\in \overline{\{y\}}$. There are three cases to consider: either $y'=y$ and $\mathrm{codim}(\overline{\{y'\}},\overline{\{y\}})=0$, or $\mathrm{codim}(\overline{\{y'\}},\overline{\{y\}})=1$, or $\mathrm{codim}(\overline{\{y'\}},\overline{\{y\}})>1$.

\underline{Case 1}: $\mathrm{codim}(\overline{\{y'\}},\overline{\{y\}})>1$. In this case, we have $(d_Y\circ f_*)_{y'}^x=0$ by the definition of $d_Y$. To see that $(f_*\circ d_X)_{y'}^x=0$, we note that the morphism $f$ is of finite type so the dimension formula \cite[\href{https://stacks.math.columbia.edu/tag/02JU}{Tag 02JU}]{stacks-project} implies for any $x'\in X$ such that $\mathrm{codim}(\overline{\{x'\}},X)=1$ and $f(x')=y'$ we have $\mathrm{tr}.\mathrm{deg}(\kappa(x')/\kappa(y'))>0$.

\underline{Case 2}: $\mathrm{codim}(\overline{\{y'\}},\overline{\{y\}})\leq 1$. The proof of the claim for this case goes along the same lines as the proof given in \cite[Proposition 49.9]{MR2427530}.
\end{proof}

Assume now that $X$ is Noetherian. For any flat morphism $f:X\rightarrow Y$ there is a flat pullback \[f^*:C(Y)\rightarrow C(X)\] defined as in \cite[\S49.D]{MR2427530}. For $y\in Y$ and $x\in X$, the map $f^*$ is defined componentwise by \[(f^*)^y_x:K^M_n(\kappa(y))\rightarrow K^M_m(\kappa(x))\]
to be trivial unless $m=n$, $f(x)=y$, and $x\in X_y$ is a generic point. If these conditions are met, then one sets $(f^*)^y_x=\mathrm{length}(\mathcal{O}_{X_y,x})\cdot r_{\kappa(x)/\kappa(y)}$ where $r_{\kappa(x)/\kappa(y)}$ is the restriction map in Milnor $K$-theory. Note that the assumption $X$ is Noetherian is necessary for the definition of the pullback, since it implies that the fiber over any point $y\in Y$ has at most finitely many irreducible components. 

The flat pullback is functorial for flat morphisms of Noetherian schemes (see \cite[Proposition 49.18]{MR2427530} and use that flat maps lift generalizations \cite[\href{https://stacks.math.columbia.edu/tag/03HV}{Tag 03HV}]{stacks-project}). Moreover, given a Cartesian diagram \[\begin{tikzcd}X'=X\times_Y Y'\arrow["g'"]{d}\arrow["f'"]{r} & Y'\arrow["g"]{d} \\ X\arrow["f"]{r} & Y \end{tikzcd}\] such that the following conditions hold: \begin{enumerate}
	\item $Y'$ is Noetherian, 
	\item $g$ is flat, 
	\item $f$ is finite,
\end{enumerate} then it follows (from the proof of \cite[Proposition 49.20]{MR2427530}) that diagram \[\begin{tikzcd}C(X)\arrow["(g')^*"]{d}\arrow["f'_*"]{r} & C(Y)\arrow["g^*"]{d} \\ C(X')\arrow["f_*"]{r} & C(Y') \end{tikzcd}\] is commutative. One uses this observation to check the following:

\begin{lem}\label{lem: flatpull}
Let $X$ and $Y$ be separated and excellent schemes. Let $f:X\rightarrow Y$ be a flat morphism of schemes. Assume that $X$ is Noetherian. Assume also that $f$ is weakly catenarious. Then the flat pullback commutes with the endomorphisms $d_X$ and $d_Y$, i.e.\ $d_X\circ f^*= f^*\circ d_Y$.
\end{lem}

\begin{proof}
Let $y\in Y$ and $x'\in X$ be points. Set $y'=f(x')$. If $y'\notin\overline{\{y\}}$ then we have $(f^*\circ d_Y)_{x'}^y=0$, by the definition of $d_Y$, and $(d_X\circ f^*)^y_{x'}=0$ since $\overline{X_y}=f^{-1}(\overline{\{y\}})$ as $f$ is flat. This proves the claim in this case.

We can thus assume $y'\in \overline{\{y\}}$. There are three cases to consider: either $y'=y$ and $\mathrm{codim}(\overline{\{y'\}},\overline{\{y\}})=0$, or $\mathrm{codim}(\overline{\{y'\}},\overline{\{y\}})=1$, or $\mathrm{codim}(\overline{\{y'\}},\overline{\{y\}})>1$.

\underline{Case 1}: $\mathrm{codim}(\overline{\{y'\}},\overline{\{y\}})>1$. In this case, we have $(f^*\circ d_Y)_{x'}^y=0$ by the definition of $d_Y$. Since $f$ is assumed weakly catenarious, we also have $(d_X\circ f^*)_{y'}^x=0$.

\underline{Case 2}: $\mathrm{codim}(\overline{\{y'\}},\overline{\{y\}})\leq 1$. The proof of the claim for this case goes along the same lines as the proof given in \cite[Proposition 49.23]{MR2427530}.
\end{proof}

For any such scheme $X$ over a base field $k$, one can now check that the endomorphism $d_X$ is square-zero, i.e.\ $(d_X)^2=d_X\circ d_X=0$, following the proof of \cite[Proposition 49.30]{MR2427530} with minor modifications.

\begin{prop}
Let $X$ be a separated and excellent scheme over a field $k$. Then $(d_X)^2=0$.
\end{prop}

\begin{proof}
We check that for any two points $x,x'\in X$ we have $(d_X\circ d_X)_{x'}^x=0$. 

\underline{Step 0}: Note that it suffices to assume $X=\mathrm{Spec}(\mathcal{O}_{\overline{\{x\}},\overline{\{x'\}}})$, i.e.\ we can assume that $X$ is a $2$-dimensional excellent local ring over a field $k$.

\underline{Step 1}: Up to replacing $X$ by its normalization $X^\nu$, we can assume that $X$ is a $2$-dimensional semi-local excellent and normal scheme over a field $k$. To see this, we use that $f:X^\nu\rightarrow X$ is finite, since $X$ is assumed excellent. Thus we can apply Lemma \ref{lem: fincat} to obtain $f_*\circ (d_{X^\nu})^2=(d_X)^2\circ f_*$. Lastly $(d_X)^2\circ f_*=(d_X\circ d_X)_{x'}^x$ since $X$ is assumed irreducible and local and $X$ and $X^\nu$ have the same fraction field. After possibly repeating \textit{Step 0}, we can assume again that $X$ is the spectrum of a local ring $(S,\mathfrak{m}_S)$ with $S$ normal.

\underline{Step 2}: Since $S$ is a normal, local, Noetherian ring with $\mathrm{dim}(\mathrm{Spec}(S))=2$, it follows $S$ is Cohen-Macaulay, see \cite[\href{https://stacks.math.columbia.edu/tag/0B3D}{Tag 0B3D}]{stacks-project}. Let $\widehat{S}$ be the completion of $S$ at $\mathfrak{m}_S$. 
Then the completion $\widehat{S}$ is also a Cohen-Macaulay, Noetherian, local ring by \cite[\href{https://stacks.math.columbia.edu/tag/07NX}{Tag 07NX}]{stacks-project}. Hence $\mathrm{Spec}(\widehat{S})$ is equidimensional with $\mathrm{dim}(\mathrm{Spec}(\widehat{S}))=2$. 

The map induced on spectra from the morphism $S\rightarrow \widehat{S}$ is flat and, by Example \ref{exmp: localweakcat}, weakly catenarious. By Lemma \ref{lem: flatpull}, it thus suffices to check the claim for $X=\mathrm{Spec}(\widehat{S})$.

\underline{Step 3}: As in \cite[Proposition 49.30, Step 2]{MR2427530}, one can reduce from the case of $X=\mathrm{Spec}(\widehat{S})$ to the case where $X=\mathrm{Spec}(F[[x,y]])$ for a field $F/k$ using Lemma \ref{lem: properpush}.

\underline{Step 4}: The proof for the scheme $X=\mathrm{Spec}(F[[x,y]])$ is given in \cite[Proposition 49.30, Step 1]{MR2427530}.
\end{proof}

Amazingly, we note that all of the above follows without keeping track of any additional gradings on the group $C(X)$ which the endomorphism $d_X$ may or may not respect.

Now write $X^{(p)}$ for the set of points $x\in X$ of codimension $p$. For any pair of integers $p,q\in \mathbb{Z}$ we write \[C^p_q(X):= \bigoplus_{x\in X^{(p)}} K_{q-p}(\kappa(x))\] for the given direct sum. Since $X$ is locally Noetherian, there is a decomposition $C(X)=\bigoplus_{(p,q)\in \mathbb{Z}^2} C^p_{q}(X)$ giving $C(X)$ the structure of a bigraded group. 

If $X$ is irreducible, or if $X$ has finite dimension and is biequidimensional, then $d_X$ has bidegree $(1,0)$ on $C(X)$ with this grading. Indeed, because $X$ is excellent, the underlying topological space of $X$ is catenary. Thus if $x,x'\in X$ are such that $\overline{\{x'\}}\subset \overline{\{x\}}$ and $\mathrm{codim}(\overline{\{x'\}}, \overline{\{x\}})=1$, then \[\mathrm{codim}(\overline{\{x'\}},X)=\mathrm{codim}(\overline{\{x\}},X)+\mathrm{codim}(\overline{\{x'\}},\overline{\{x\}})=\mathrm{codim}(\overline{\{x\}},X)+1\] by \cite[\href{https://stacks.math.columbia.edu/tag/02I6}{Tag 02I6}]{stacks-project} in the case $X$ is irreducible and by \cite[Proposition 5.3]{MR3631826} in the case $X$ has finite dimension and is biequidimensional.

\begin{defn}
	Let $k$ be a field. Let $X/k$ be a separated and excellent scheme. Assume that $X$ is either irreducible or finite dimensional and biequidimensional. For any $p,q\in \mathbb{Z}$, we define \textit{the degree-$(p,q)$ $K$-cohomology group} $A^p(X;K^M_q)$ to be the homology of the complex $C^*_q(X)$ induced by $d_X$ in degree $p$, i.e.\ \[A^p(X;K^M_q):=\HH\left(\bigoplus_{x\in X^{(p-1)}} K^M_{q-p+1}(\kappa(x))\rightarrow \bigoplus_{x\in X^{(p)}} K^M_{q-p}(\kappa(x)) \rightarrow \bigoplus_{x\in X^{(p+1)}} K^M_{q-p-1}(\kappa(x)) \right).\]
\end{defn}

This definition seems to be the most suitable for our purposes as the next corollary shows.

\begin{cor}\label{cor: gerst}
	Let $k$ be an infinite field and let $X$ be a separated, excellent, and regular scheme over $k$. Assume that $X$ is either irreducible or finite dimensional and biequidimensional. 
	
	Then there is an isomorphism \[\mathrm{H}^p(X,\mathcal{K}^M_q)\cong A^p(X;K^M_q)\] where $\mathcal{K}_q^M$ is the $q$th Milnor $K$-theory sheaf on $X$.
\end{cor}

\begin{proof}
	The differential $d_X$ is the same as that coming from the Gersten resolution. The claim is immediate from the exactness of this resolution \cite[Theorem 7.1]{MR2461425}.
\end{proof}

\subsection{Functorality of $K$-cohomology}
We collect here some of the possible functorial transformations that are available for the $K$-cohomology groups. We do not use this material in the proofs of our main results, but it may be worthwhile to make note of it.

Let $X$ and $Y$ be separated, locally Noetherian, and excellent schemes. We say that a morphism $f:X\rightarrow Y$ has \textit{relative codimension $d$} if for any $p$, and for any point $x\in X^{(p)}$, the image $f(x)=y$ lies in $Y^{(p+d)}$.

\begin{exmp}\label{exmp: local}
	Any open immersion to $X$ has relative codimension 0. So does the canonical morphism from the spectrum of the local ring of a point $x\in X$ to $X$.
\end{exmp}

\begin{exmp}\label{exmp: subv}
	Suppose that $X$ is irreducible and that either of the following are true:
	\begin{enumerate}
		\item $Y$ is irreducible;
		\item $Y$ is finite dimensional and biequidimensional.
	\end{enumerate} Then, if $f:X\rightarrow Y$ is a closed immersion, it follows that $f$ must have relative codimension $d=\mathrm{codim}(X,Y)$. To see this, first assume $Y$ is irreducible. If $x\in X^{(p)}$ and $Z=\overline{\{x\}}$ then we have  \[\mathrm{codim}(Z,X)+\mathrm{codim}(X,Y)=\mathrm{codim}(Z,Y)\] by \cite[\href{https://stacks.math.columbia.edu/tag/02I6}{Tag 02I6}]{stacks-project} since $Y$ is assumed catenary (which follows from being excellent).
	
	Now assume that $Y$ is finite dimensional and biequidimensional. Then the same equality on codimension as above follows from \cite[Proposition 2.3]{MR3631826}.
\end{exmp}

\begin{exmp}\label{exmp: normv}
	Suppose that $Y$ is an integral scheme and $f:Y^\nu\rightarrow Y$ is the normalization morphism of $Y$. Then $f$ has relative codimension 0. To see this, let $x\in (Y^\nu)^{(p)}$ and set $y=f(x)$. Assume that $y\in Y^{(p')}$. Then there exists a chain $\mathcal{C}$ of irreducible closed subsets of $Y$ of length $p'$ and the going down theorem guarantees a chain $\mathcal{D}$ of irreducible closed subsets of $Y^{\nu}$ of length $p'$, starting with $x$, which maps to $\mathcal{C}$ under $f$; so $p\geq p'$. Conversely, if $\mathcal{D}$ is a chain of closed subsets of $Y^{\nu}$ of length $p$ containing $x$, then $\mathcal{C}=f(\mathcal{D})$ is a chain of closed subsets of $Y$ of length $p$ containing $y$ by \cite[\href{https://stacks.math.columbia.edu/tag/00GT}{Tag 00GT}]{stacks-project}. So $p\leq p'$.
\end{exmp}

\begin{exmp}\label{exmp: finiteflat}
	Suppose that $X$ and $Y$ have finite dimension and are biequidimensional. Then every finite, flat, and surjective morphism $f:X\rightarrow Y$ is of relative codimension 0. Indeed, let $y\in Y$ be a point and set $Z=\overline{\{y\}}$. Then \[\mathrm{codim}(Z,Y)=\mathrm{codim}(f^{-1}(Z),X)\] by \cite[Corollaire 6.1.4]{MR0199181}. Now let $W\subset f^{-1}(Z)$ be an irreducible component. Then $W$ dominates $Z$ since generalizations lift along the flat morphism $f$ (see \cite[\href{https://stacks.math.columbia.edu/tag/03HV}{Tag 03HV}]{stacks-project}). Since $f$ is finite (so closed), this implies $W$ surjects onto $Z$. Hence by \cite[\href{https://stacks.math.columbia.edu/tag/0ECG}{Tag 0ECG}]{stacks-project}, we have $\mathrm{dim}(W)=\mathrm{dim}(Z)$ and also $\mathrm{dim}(X)=\mathrm{dim}(Y)$. Now $\mathrm{codim}(W,X)=\mathrm{codim}(Z,Y)$ follows from \cite[Proposition 2.3]{MR3631826}.
\end{exmp}

If $f:X\rightarrow Y$ is a proper morphism of relative codimension $d$ (which necessarily implies that $f$ is finite), then the proper pushforward induces a pushforward map
\[f_*:C^p_q(X)\rightarrow C^{p+d}_{q+d}(Y)\] for any $p,q\in \mathbb{Z}$. In other words, the proper pushforward of a morphism of relative codimension $d$ is a morphism of bidegree $(d,d)$ between the bigraded groups $C(X)$ and $C(Y)$. If $X$ and $Y$ are two schemes over a field $k$, and if both $X$ and $Y$ are either irreducible or of finite dimension and biequidimensional, then by Lemma \ref{lem: properpush} this induces a pushforward map \[f_*:A^p(X;K^M_q)\rightarrow A^{p+d}(Y;K^M_{q+d})\] which is functorial for compositions of proper morphisms of fixed relative codimension.

Now assume that $X$ is, moreover, a Noetherian scheme. We say a morphism $f:X\rightarrow Y$ has \textit{continuous codimension} if for any point $y\in Y^{(p)}$, every generic point $x$ of the fiber $f^{-1}(y)$ satisfies $x\in X^{(p)}$. We allow for the possibility that the fiber $f^{-1}(y)$ is empty.

\begin{exmp}\label{exmp: formalfibers}
	Let $\mathrm{Spec}(R)$ be the spectrum of a $2$-dimensional regular local ring $R$ with maximal ideal $\mathfrak{m}$ and let $\mathrm{Spec}(\widehat{R})$ be the spectrum of the completion of $R$ at $\mathfrak{m}$. Then the canonical map \[\Lambda:\mathrm{Spec}(\widehat{R})\rightarrow \mathrm{Spec}(R)\] is flat and of continuous codimension.
	
	To see this, one can work in cases, checking points of the fiber over a point $\mathfrak{p}\subset R$ with $\mathrm{ht}(\mathfrak{p})=0,1,2$. For $\mathrm{ht}(\mathfrak{p})=2$ (so $\mathfrak{m}=\mathfrak{p}$), if there is prime $\mathfrak{q}\subset \widehat{R}$ with $\mathfrak{q}\cap R=\mathfrak{p}$ then $\mathfrak{q}$ can't be $(0)$, so it has height either one or two. By the going down theorem, the height also can't be one. Hence $\mathfrak{q}$ is the maximal ideal of $\widehat{R}$. For $\mathrm{ht}(\mathfrak{p})=1$, the only possibility for the fiber are primes of height one as well. Lastly, the generic fiber of $\Lambda$ always contains the generic point of the domain, hence the claim.
\end{exmp}

\begin{exmp}\label{exmp: contlocal}
	Any morphism of relative codimension $0$ has continuous codimension. The converse is not true, however, as the projection morphism from a variety of positive dimension gives a counterexample.
	
	For a more interesting example of a morphism which has continuous codimension but is not of relative codimension $0$, take $R=k[x,y]_{(x,y)}$ and let $\widehat{R}$ be the completion of $R$ at $\mathfrak{m}=(x,y)$. Then the morphism of schemes associated to the inclusion $R\subset \widehat{R}$ fails to have relative codimension $0$ by \cite[Theorem 2]{MR0977763} but, this morphism does have continuous codimension by the previous example.
\end{exmp}

The flat pullback along a flat morphism $f:X\rightarrow Y$ of continuous codimension induces a pullback \[f^*:C_q^p(Y)\rightarrow C_q^p(X)\] for any $p,q\in \mathbb{Z}$. Thus the flat pullback along a morphism $f^*$ induces a bigraded morphism of degree $(0,0)$ between the bigraded groups $C(Y)$ and $C(X)$. If $X$ and $Y$ are two schemes over a field $k$, and if both $X$ and $Y$ are also either irreducible or of finite dimension and biequidimensional, then such maps are weakly catenarious and, by Lemma \ref{lem: flatpull}, there is an induced pullback morphism \[f^*:A^p(Y;K^M_q)\rightarrow A^p(X;K^M_q)\] which is functorial for compositions of flat morphisms of continuous codimension between Noetherian schemes.

A priori, it isn't clear that morphisms of fixed relative codimension or of continuous codimension are stable under base change. The following gives a special case of this.

\begin{lem}\label{lem: basechange}
	Suppose that $f:X\rightarrow Y$ is a flat morphism with continuous codimension. Let $\pi:Y'\rightarrow Y$ be a morphism of relative codimension $d$ and assume that the projection $\pi':X'=X\times_Y Y'\rightarrow X$ also has relative codimension $d$. Then the base change $f':X'\rightarrow Y'$ is a flat morphism of continuous codimension.
	
	Moreover, if $\pi$ is finite and if $X$ and $X'$ are Noetherian, then the diagram \begin{equation}\label{eq: trans}\begin{tikzcd} C^p_q(Y')\arrow["(f')^*"]{r}\arrow["\pi_*"]{d}& C^p_q(X')\arrow["\pi'_*"]{d} \\ C_{q+d}^{p+d}(Y) \arrow["f^*"]{r} & C_{q+d}^{p+d}(X) \end{tikzcd}\end{equation} is commutative.
\end{lem}

\begin{proof}
	Let $y\in (Y')^{(p)}$ be a point. Because generalizations lift along flat maps \cite[\href{https://stacks.math.columbia.edu/tag/03HV}{Tag 03HV}]{stacks-project}, for any $w\in(f')^{-1}(y)$ we have $\mathrm{codim}(\overline{\{w\}}, X')\geq p$.
	
	Let $z=\pi(y)\in Y^{(p+d)}$ and let $u\in f^{-1}(z)$ be any generic point so that $u\in X^{(p+d)}$. Then there is a point $u'\in (f')^{-1}(y)$ which maps to $u$ by base change. Since $\pi'$ has relative codimension $d$, we must have $u'\in (X')^{(p)}$. Now $u'$ is a specialization of any generic point $w$, so $\mathrm{codim}(\overline{\{w\}}, X')\leq p$ as well.
	
	That the diagram (\ref{eq: trans}) is commutative follows from the fact that it is induced from the corresponding commutative diagram between the ungraded groups $C(Y'), C(X'), C(Y), C(X)$.
\end{proof}

Under the assumptions of Lemma \ref{lem: basechange}, we thus find the commuting square \begin{equation}\begin{tikzcd} A^p(Y';K^M_q)\arrow["(f')^*"]{r} \arrow["\pi_*"]{d} & A^p(X';K_q^M)\arrow["\pi'_*"]{d}\\ A^{p+d}(Y;K^M_{q+d})\arrow["f^*"]{r} & A^{p+d}(X;K^M_{q+d})\end{tikzcd}\end{equation} granted that all of these groups are defined.

\subsection{Relative $K$-homology}
Now we set-up a relative theory similar to \cite[\S20.1]{MR1644323}. We fix throughout the following a separated, Noetherian, excellent scheme $S$ of finite dimension defined over a fixed field $k$.

Let $X$ be a scheme that's separated and of finite type over $S$. We write $(X/S)_{(p)}$ for the set of points of $X$ whose closure has relative dimension $p$ over $S$. Following \cite[\S20.1]{MR1644323}, we say that an integral subscheme $V\subset X$ has relative dimension $\mathrm{dim}_S(V)\in \mathbb{Z}$ where \[\mathrm{dim}_S(V)=\mathrm{tr}.\mathrm{deg}(R(V)/R(T))-\mathrm{codim}(T,S);\] here $T$ is the closure of the image of $V$ in $S$ and $R(V),R(T)$ are the corresponding function fields.

We write \[C_{p,q}(X/S):= \bigoplus_{x\in (X/S)_{(p)}} K^M_{p+q}(\kappa(x)).\] The decomposition $C(X)=\bigoplus_{p,q\in \mathbb{Z}} C_{p,q}(X/S)$ gives $C(X)$ the structure of a bigraded group. According to \cite[Lemma 20.1 (2)]{MR1644323}, the endomorphism $d_X$ has degree $(-1,0)$ on $C(X)$ with this grading.

\begin{defn}
We define \textit{the degree-$(p,q)$ relative $K$-homology} group $A_p(X/S;K^M_q)$ as the homology of the complex $C_{*,q}(X/S)$ induced by the differential $d_X$ in degree $p$, i.e.\ \[A_p(X/S;K^M_q):=\HH\left(\bigoplus_{x\in (X/S)_{(p+1)}} K^M_{p+q+1}(\kappa(x))\rightarrow \bigoplus_{x\in (X/S)_{(p)}} K^M_{p+q}(\kappa(x)) \rightarrow \bigoplus_{x\in (X/S)_{(p-1)}} K^M_{p+q-1}(\kappa(x)) \right).\] Note that we don't assume $X$ is irreducible or finite dimensional and biequidimensional in this definition.
\end{defn}

\begin{rmk}
Note that the set $(X/S)_{(p)}$ is not necessarily empty if $p<0$ and so $A_p(X/S;K^M_q)$ does not necessarily vanish for $p<0$. For example, if $X$ is a positive dimensional projective variety over a field $k$, then $A_p(X/X;K^M_{-p})\cong \mathrm{CH}_{\mathrm{dim}(X)+p}(X)$ which is never zero for $0\geq p \geq -\mathrm{dim}(X)$.

As another example, if we consider $T$ from Remark \ref{rmk: a1} as a scheme over the spectrum of $k[[x]]$, then the maximal ideal $\mathfrak{m}'=(x,y)$ defines a closed subscheme of $T$ with relative dimension $\mathrm{dim}_{k[[x]]}(\mathfrak{m}')=-1$. For comparison, we also have $\mathrm{dim}_{k[[x]]}(T)=1$ and $\mathrm{dim}_{k[[x]]}(\mathfrak{m})=0$.
\end{rmk}

If $f:X\rightarrow Y$ is a proper $S$-morphism between two schemes $X$ and $Y$ which are both of finite type over $S$, then the proper pushforward $f_*:C(X)\rightarrow C(Y)$ respects the grading by relative dimension by \cite[Lemma 20.1 (3)]{MR1644323}. In particular, there are induced pushforward morphisms \[f_*:C_{p,q}(X)\rightarrow C_{p,q}(Y).\] By Lemma \ref{lem: properpush}, these pushforwards induce functorial pushforward morphisms \[f_*:A_p(X/S;K^M_q)\rightarrow A_p(Y/S;K^M_q)\] for any $p,q\in \mathbb{Z}$. 

If $f:X\rightarrow Y$ is, instead, a flat $S$-morphism of relative dimension $d$ between finite type schemes over $S$, then the flat pullback $f^*:C(Y)\rightarrow C(X)$ induces a morphism \[f^*:C_{p,q}(Y)\rightarrow C_{p+d,q-d}(X).\] Here a flat morphism $f:X\rightarrow Y$ is said to have relative dimension $d$ if for every $y\in Y$ the fiber $X_y$ is either empty or equidimensional with $\mathrm{dim}(X_y)=d$. To see that $f^*$ has target as claimed, it suffices to assume that $X$ and $Y$ are both integral. The claim then follows from \cite[Lemma 20.1 (3)]{MR1644323}. Any such morphism $f$ is weakly catenarious, thus there are functorial flat pullback morphisms \[f^*:A_p(Y/S;K_q^M)\rightarrow A_{p+d}(X/S;K^M_{q-d})\] for morphisms of fixed relative dimension by Lemma \ref{lem: flatpull}. 

\begin{rmk}\label{rmk: flatreldim}
	Consider the scheme $T$ from Remark \ref{rmk: a1}. Then $T$ can be canonically realized as an open subscheme of $\mathbb{P}^1_{k[[x]]}$ as the complement of infinity. The point $\mathfrak{m}\in T$ dominates a closed subscheme $W\subset\mathbb{P}^1_{k[[x]]}$ with $\mathrm{dim}(W)=1$. So, if one uses the definition of a flat morphism of constant relative dimension as given in \cite[\S49.D]{MR2427530}, then the open immersion of $T$ into $\mathbb{P}^1_{k[[x]]}$ is not of constant relative dimension.
	
	In contrast, with the definition we've given above, every open immersion of finite type schemes over $S$ is of constant relative dimension $0$ (and, in the above example, $\mathrm{dim}_{k[[x]]}(W)=0$).
\end{rmk}

\begin{rmk}\label{rmk: cohho}
Note that if $X$ is irreducible and of finite type over $S$ with relative dimension $\dim_S(X)=n$, then there are equalities \[C^p_q(X)=C_{n-p,q-n}(X/S)\quad \mbox{and} \quad A^p(X;K^M_q)=A_{n-p}(X/S;K^M_{q-n})\] for any $p,q\in \mathbb{Z}$ by \cite[Lemma 20.1 (2)]{MR1644323}.
\end{rmk}

Given an open subscheme $U\subset X$ of finite type over $S$ with closed complement $Z\subset X$, there is an associated long exact localization sequence in relative $K$-homology. The projective bundle formula holds in relative $K$-homology. Gysin morphisms are also defined for local complete intersection morphisms between schemes of finite type over $S$ which factor as the composition of a regular closed embedding and a smooth morphism. We couldn't find a reference for the blow-up formula in $K$-(co)homology, but we want to use it, so we include it below in this relative setting.

\begin{prop}\label{prop: blowup}
Let $S$ be a separated, Noetherian, excellent scheme of finite dimension over a field $k$. Let $X/S$ be a scheme smooth over $S$. Let $Z$ be either an irreducible or biequidimensional scheme over $S$ and let $i:Z\rightarrow X$ be a regular closed embedding of codimension $d$ over $S$. Then there are isomorphisms \[ A_p(\mathrm{Bl}_Z(X)/S; K_q^M)\cong A_p(X/S; K_q^M)\oplus \left(\bigoplus_{i=1}^{d-1} A_{p-i}(Z/S;K_{q+i}^M)\right)\] for every $p,q\in \mathbb{Z}$.
\end{prop}

\begin{proof}
Let $\pi: \mathrm{Bl}_Z(X)\rightarrow X$ be the blow-down map over $S$. Note $\pi$ is a projective local complete intersection morphism. In particular $\pi$ factors into the composition of maps (for some not necessarily unique $r>0$) \[\pi:\mathrm{Bl}_Z(X)\xrightarrow{i} \mathbb{P}^r_S\times_S X\xrightarrow{p} X.\] In this setting, there is a pullback $\pi^*: A_p(X/S; K_q^M)\rightarrow A_p(\mathrm{Bl}_Z(X)/S; K_q^M)$ defined as $\pi^*=i^!\circ p^*$ where $i^!$ is the associated Gysin map and $p^*$ is the flat pullback. Let $\psi:E\rightarrow Z$ be the projective bundle of the normal bundle associated to the immersion $Z\rightarrow X$ and write $j:E\rightarrow \mathrm{Bl}_Z(X)$ for the induced immersion. By base change, the map $\psi$ factors \[\psi:E\xrightarrow{j} \mathbb{P}_S^r\times_S Z\xrightarrow{q} Z\] and $\psi^*=j^!\circ q^*$ is defined as well. Lastly, we let $U=X\setminus Z$ and $U'=\mathrm{Bl}_Z(X)\setminus E$.

We then have the following commutative ladder of long exact localization sequences. \[
\begin{tikzcd}[column sep =small] \cdots \arrow{r} & A_{p}(E/S;K_q^M)\arrow{r} & A_{p}(\mathrm{Bl}_Z(X)/S;K_q^M)\arrow{r} & A_{p}(U'/S;K_q^M)\arrow{r} & A_{p-1}(E/S;K_q^M)\arrow{r} & \cdots\\
\cdots \arrow{r} & A_{p}(Z/S;K_q^M)\arrow{r}\arrow["e(F)\circ \psi^*"]{u} & A_{p}(X/S;K_q^M)\arrow{r}\arrow["\pi^*"]{u} & A_{p}(U/S;K_q^M)\arrow{r}\arrow[equals]{u} & A_{p-1}(Z/S;K_q^M)\arrow{r}\arrow["\psi^*"]{u} & \cdots \end{tikzcd}\] Here $e(F)$ is the Euler class of the excess intersection bundle $F$ associated to the square with horizontal maps $j$ and $i$, cf.\ \cite[Proposition 55.3]{MR2427530}. The bundle $F$ can, moreover, be identified with the universal quotient sheaf of the projective bundle $\psi$ by the arguments of \cite[\href{https://stacks.math.columbia.edu/tag/0FV9}{Tag 0FV9}]{stacks-project} and \cite[\href{https://stacks.math.columbia.edu/tag/0FVA}{Tag 0FVA}]{stacks-project}.

One checks that the following standard sequence associated to the above diagram is exact.
\[\cdots\rightarrow A_p(Z/S;K_q^M)\xrightarrow{(e(F)\circ \psi^*,-i_*)}A_p(E/S;K_q^M)\oplus A_p(X/S;K_q^M)\xrightarrow{j_*+\pi^*} A_p(\mathrm{Bl}_Z(X)/S;K_q^M)\rightarrow A_{p-1}(Z/S;K_q^M)\rightarrow \cdots\] The composition $\psi_*\circ e(F)\circ \psi^*$ is the identity and so this sequence breaks into short exact sequences \[0\rightarrow A_p(Z/S;K^M_q)\rightarrow A_p(E/S;K_q^M)\oplus A_p(X/S;K_q^M)\rightarrow A_p(\mathrm{Bl}_Z(X)/S);K_q^M)\rightarrow 0.\] The claim follows now from the projective bundle formula applied to $\psi:E\rightarrow Z$.
\end{proof}

	\bibliographystyle{amsalpha}
	\bibliography{bib}
\end{document}